\documentclass[a4paper,12pt,reqno]{amsart}
\usepackage{amsfonts}
\usepackage{amsmath}
\usepackage{amssymb}
\usepackage[a4paper]{geometry}
\usepackage{mathrsfs}
\usepackage[colorlinks]{hyperref}
\renewcommand\eqref[1]{(\ref{#1})}
\makeatletter
\newcommand*{\mint}[1]{%
  \mint@l{#1}{}%
}
\newcommand*{\mint@l}[2]{%
  \@ifnextchar\limits{%
    \mint@l{#1}%
  }{%
    \@ifnextchar\nolimits{%
      \mint@l{#1}%
    }{%
      \@ifnextchar\displaylimits{%
        \mint@l{#1}%
      }{%
        \mint@s{#2}{#1}%
      }%
    }%
  }%
}
\newcommand*{\mint@s}[2]{%
  \@ifnextchar_{%
    \mint@sub{#1}{#2}%
  }{%
    \@ifnextchar^{%
      \mint@sup{#1}{#2}%
    }{%
      \mint@{#1}{#2}{}{}%
    }%
  }%
}
\def\mint@sub#1#2_#3{%
  \@ifnextchar^{%
    \mint@sub@sup{#1}{#2}{#3}%
  }{%
    \mint@{#1}{#2}{#3}{}%
  }%
}
\def\mint@sup#1#2^#3{%
  \@ifnextchar_{%
    \mint@sup@sub{#1}{#2}{#3}%
  }{%
    \mint@{#1}{#2}{}{#3}%
  }%
}
\def\mint@sub@sup#1#2#3^#4{%
  \mint@{#1}{#2}{#3}{#4}%
}
\def\mint@sup@sub#1#2#3_#4{%
  \mint@{#1}{#2}{#4}{#3}%
}
\newcommand*{\mint@}[4]{%
  \mathop{}%
  \mkern-\thinmuskip
  \mathchoice{%
    \mint@@{#1}{#2}{#3}{#4}%
        \displaystyle\textstyle\scriptstyle
  }{%
    \mint@@{#1}{#2}{#3}{#4}%
        \textstyle\scriptstyle\scriptstyle
  }{%
    \mint@@{#1}{#2}{#3}{#4}%
        \scriptstyle\scriptscriptstyle\scriptscriptstyle
  }{%
    \mint@@{#1}{#2}{#3}{#4}%
        \scriptscriptstyle\scriptscriptstyle\scriptscriptstyle
  }%
  \mkern-\thinmuskip
  \int#1%
  \ifx\\#3\\\else_{#3}\fi
  \ifx\\#4\\\else^{#4}\fi
}
\newcommand*{\mint@@}[7]{%
  \begingroup
    \sbox0{$#5\int\m@th$}%
    \sbox2{$#5\int_{}\m@th$}%
    \dimen2=\wd0 %
    \let\mint@limits=#1\relax
    \ifx\mint@limits\relax
      \sbox4{$#5\int_{\kern1sp}^{\kern1sp}\m@th$}%
      \ifdim\wd4>\wd2 %
        \let\mint@limits=\nolimits
      \else
        \let\mint@limits=\limits
      \fi
    \fi
    \ifx\mint@limits\displaylimits
      \ifx#5\displaystyle
        \let\mint@limits=\limits
      \fi
    \fi
    \ifx\mint@limits\limits
      \sbox0{$#7#3\m@th$}%
      \sbox2{$#7#4\m@th$}%
      \ifdim\wd0>\dimen2 %
        \dimen2=\wd0 %
      \fi
      \ifdim\wd2>\dimen2 %
        \dimen2=\wd2 %
      \fi
    \fi
    \rlap{%
      $#5%
        \vcenter{%
          \hbox to\dimen2{%
            \hss
            $#6{#2}\m@th$%
            \hss
          }%
        }%
      $%
    }%
  \endgroup
}
\numberwithin{equation}{section}
\theoremstyle{plain}
\newtheorem{thm}{Theorem}[section]
\newtheorem{prop}[thm]{Proposition}

\theoremstyle{definition}
\newtheorem{defn}[thm]{Definition}
\newtheorem{rem}[thm]{Remark}

\begin{document}

   \title[HLS and Stein-Weiss  inequalities on homogeneous Lie groups]
{Hardy-Littlewood-Sobolev and Stein-Weiss inequalities on homogeneous Lie groups}

\author[A. Kassymov]{Aidyn Kassymov}
\address{
  Aidyn Kassymov:
  \endgraf
  \endgraf
  Institute of Mathematics and Mathematical Modeling
  \endgraf
  125 Pushkin str.
  \endgraf
  050010 Almaty
  \endgraf
  Kazakhstan
  \endgraf
  and
  \endgraf
  Al-Farabi Kazakh National University
  \endgraf
   71 Al-Farabi avenue
   \endgraf
   050040 Almaty
   \endgraf
   Kazakhstan
  {\it E-mail address} {\rm kassymov@math.kz}}

\author[M. Ruzhansky]{Michael Ruzhansky}
\address{
  Michael Ruzhansky:
  \endgraf
  Department of Mathematics
  \endgraf
  Ghent University, Belgium
  \endgraf
  and
  \endgraf
  School of Mathematical Sciences
    \endgraf
    Queen Mary University of London
  \endgraf
  United Kingdom
  \endgraf
  {\it E-mail address} {\rm ruzhansky@gmail.com}
  }

\author[D. Suragan]{Durvudkhan Suragan}
\address{
	Durvudkhan Suragan:
	\endgraf
	Department of Mathematics
	\endgraf
	School of Science and Technology, Nazarbayev University
	\endgraf
	53 Kabanbay Batyr Ave, Astana 010000
	\endgraf
	Kazakhstan
	\endgraf
	{\it E-mail address} {\rm durvudkhan.suragan@nu.edu.kz}}
\thanks{
The authors were supported in parts by the FWO Odysseus Project, by the EPSRC grant EP/R003025/1 and by the Leverhulme
Grant RPG-2017-151, as well as by the MESRK grant AP05130981.
\
}

     \keywords{Riesz potential, Hardy-Littlewood-Sobolev inequality, Stein-Weiss inequality, fractional operator, fractional integral,
     	 homogeneous Lie group.}
 \subjclass{22E30, 43A80.}

     \begin{abstract}
In this note we prove the Stein-Weiss inequality on general homogeneous Lie groups. The obtained results extend previously known inequalities. Special properties of homogeneous norms play a key role in our proofs. Also, we give a simple proof of the Hardy-Littlewood-Sobolev inequality on general homogeneous Lie groups.
     \end{abstract}
     \maketitle

\section{Introduction}

Historically, in \cite{HL28}, Hardy and Littlewood considered the one dimensional fractional integral operator on $(0,\infty)$ given by
\begin{equation}\label{1Doper}
T_{\lambda}u(x)=\int_{0}^{\infty}\frac{u(y)}{|x-y|^{\lambda}}dy,\,\,\,\,0<\lambda<1,
\end{equation}
and proved the following theorem:
\begin{thm}\label{1DHLS28}
Let $1<p<q<\infty$ and $u\in L^{p}(0,\infty)$  with $\frac{1}{q}=\frac{1}{p}+\lambda-1$, then
\begin{equation}
\|T_{\lambda}u\|_{L^{q}(0,\infty)}\leq C \|u\|_{L^{p}(0,\infty)},
\end{equation}
where $C$ is a positive constant independent of $u$.
\end{thm}
The $N$-dimensional analogue of \eqref{1Doper} can be written by the formula:
\begin{equation}\label{NDoper}
I_{\lambda}u(x)=\int_{\mathbb{R}^{N}}\frac{u(y)}{|x-y|^{\lambda}}dy,\,\,\,\,0<\lambda<N.
\end{equation}
The $N$-dimensional case of  Theorem \ref{1DHLS28} was extended by Sobolev in
\cite{Sob38}:
\begin{thm}\label{THM:HLS}
Let $1<p<q<\infty$, $u\in L^{p}(\mathbb{R}^{N})$  with $\frac{1}{q}=\frac{1}{p}+\frac{\lambda}{N}-1$, then
\begin{equation}
\|I_{\lambda}u\|_{L^{q}(\mathbb{R}^{N})}\leq C \|u\|_{L^{p}(\mathbb{R}^{N})},
\end{equation}
where $C$ is a positive constant independent of $u$.
\end{thm}
Later, in \cite{StWe58} Stein and Weiss obtained the following two-weight extention of the  Hardy-Littlewood-Sobolev inequality, which is known as the Stein-Weiss inequality.
\begin{thm}\label{Classiacal_Stein-Weiss_inequality}
Let $0<\lambda<N$, $1<p<\infty$, $\alpha<\frac{N(p-1)}{p}$, $\beta<\frac{N}{q}$, $\alpha+\beta\geq0$ and $\frac{1}{q}=\frac{1}{p}+\frac{\lambda+\alpha+\beta}{N}-1$. If $1<p\leq q<\infty$, then
\begin{equation}
\||x|^{-\beta}I_{\lambda}u\|_{L^{q}(\mathbb{R}^{N})}\leq C \||x|^{\alpha}u\|_{L^{p}(\mathbb{R}^{N})},
\end{equation}
where $C$ is a positive constant independent of $u$.
\end{thm}

The Hardy-Littlewood-Sobolev inequality on the Heisenberg group was obtained by Folland and Stein in \cite{FS74}.
  In  \cite{GMS} the authors studied the Stein-Weiss inequality on the Carnot groups. Note that in \cite{HLZ} the authors also proved an analogue of the Stein-Weiss inequality on the Heisenberg groups.
In \cite{ZW} author proved Stein-Weiss inequality on product spaces. In \cite{JD} author proved the Stein-Weiss inequality on the Euclidean half-space. In the works \cite{CF}, \cite{FM}, \cite{MW}  and \cite{Per} authors were studied the regularity of fractional integrals on Euclidean spaces. In this note we first give a simple proof for the Hardy-Littlewood-Sobolev inequality on general homogeneous groups, recapturing the result of \cite[Theorem 4.1]{RY} where a much heavier machinery was used.
In the proof we follow the method of Stein and Weiss, however, special properties of  homogeneous norms of the homogeneous Lie groups play a key
role in our calculations. Furthermore, in Theorem \ref{stein-weiss3} we establish the Stein-Weiss on general homogeneous groups based on the integral Hardy inequalities established in \cite{RY}.
Of course, the obtained result recovers the previously  known results of Abelian (Euclidean), Heisenberg, Carnot groups since the class of the homogeneous Lie groups contains those and since we can work with an arbitrary homogeneous quasi-norm.
Note that in this direction systematic studies of different functional inequalities on  general homogeneous (Lie) groups were initiated by the paper \cite{RSAM}. We refer to this and other papers by the authors (e.g. \cite{RSY1}) for further discussions.

We also note that the best constant in the Hardy-Littlewood-Sobolev inequality on the Heisenberg group is now known, see Frank and Lieb \cite{FL12} (in the Euclidean case this was done earlier by Lieb in \cite{Lie83}). The expression for the best constant depends on the particular quasi-norm used and may change for a different choice of the quasi-norm.

The main results of this paper are as follows:
\begin{itemize}
\item {\bf Hardy-Littlewood-Sobolev inequality}: Let $\mathbb{G}$ be a homogeneous group of homogeneous dimension $Q$ and let $|\cdot|$ be an arbitrary homogeneous quasi-norm on $\mathbb{G}$. Let $1<p<q<\infty,\,\,0<\lambda<Q$, $\frac{1}{q}=\frac{1}{p}+\frac{\lambda}{Q}-1$. Then for all $u\in L^{p}(\mathbb{G})$ and $h\in L^{q'}(\mathbb{G})$ we have
	\begin{equation}\label{EQ:HLSi1}
	\left|\int_{\mathbb{G}}\int_{\mathbb{G}}\frac{u(y)h(x)}{|y^{-1} x|^{\lambda}}dxdy\right|\leq C\|u\|_{L^{p}(\mathbb{G})}\|h\|_{L^{q'}(\mathbb{G})},
	\end{equation}
	where $C$ is a positive constant independent of $u$ and $h$.
	
	For the formulation similar to that of Theorem \ref{THM:HLS} see
	Theorem \ref{Trsob}.
\item {\bf Stein-Weiss inequality}: Let $\mathbb{G}$ be a homogeneous group of homogeneous dimension $Q$ and let $|\cdot|$ be an arbitrary homogeneous quasi-norm on $\mathbb{G}$.
Let $0<\lambda<Q$, $1<p\leq q<\infty$, $\alpha<\frac{Q}{p'}$, $\beta<\frac{Q}{q}$, $\alpha+\beta\geq0$, $\frac{1}{q}=\frac{1}{p}+\frac{\alpha+\beta+\lambda}{Q}-1$, where $\frac{1}{p}+\frac{1}{p'}=1$ and $\frac{1}{q}+\frac{1}{q'}=1$. Then we have
\begin{equation}\label{EQ:HLSi2}
\left|\int_{\mathbb{G}}\frac{u(y)h(x)}{|x|^{\beta}|y^{-1}x|^{\lambda}|y|^{\alpha}}dxdy\right|\leq C\|u\|_{L^{p}(\mathbb{G})}\|h\|_{L^{q'}(\mathbb{G})},
\end{equation}
where $C$ is a positive constant independent of $u$ and $h$.

For the formulation similar to that of Theorem \ref{Classiacal_Stein-Weiss_inequality} see
	Theorem \ref{stein-weiss3}.
	
Although \eqref{EQ:HLSi1} is clearly contained in \eqref{EQ:HLSi2}, we still keep them as separate statements since the Hardy-Littlewood-Sobolev inequality  \eqref{EQ:HLSi1} allows for a simple proof which is much more transparent than that of the Stein-Weiss inequality \eqref{EQ:HLSi2}. The present proof is also much simpler than the original proof of \eqref{EQ:HLSi1} in
\cite{RY}.
\end{itemize}

 Finally, let us note that the heavier machinery developed in \cite{RY} also yielded a differential version of the Stein-Weiss inequality (which may be also called {Stein-Weiss-Sobolev inequality}), however, in a more special case of graded groups as follows (see \cite[Theorem 5.12]{RY} for details and the proof):

\begin{itemize}
\item {\bf Differential Stein-Weiss (or Stein-Weiss-Sobolev) inequality}:
Let $\mathbb{G}$ be a graded Lie group of homogeneous dimension $Q$ and let $|\cdot|$ be an arbitrary homogeneous quasi-norm. Let $1<p,q<\infty$, $0\leq a<Q/p$ and $0\leq b<Q/q$. Let $0<\lambda<Q$, $0\leq \alpha <a+Q/p'$ and $0\leq \beta\leq b$ be such that $(Q-ap)/(pQ)+(Q-q(b-\beta))/(qQ)+(\alpha+\lambda)/Q=2$ and $\alpha+\lambda\leq Q$, where $1/p+1/p'=1$. Then there exists a positive constant $C=C(Q,\lambda, p, \alpha, \beta, a, b)$ such that
\begin{equation}\label{HLS_ineq1_grad}
\left|\int_{\mathbb{G}}\int_{\mathbb{G}}\frac{\overline{f(x)}g(y)}{|x|^{\alpha}|y^{-1}x|^{\lambda}|y|^{\beta}}dxdy\right|\leq C\|f\|_{\dot{L}^{p}_{a}(\mathbb{G})}\|g\|_{\dot{L}^{q}_{b}(\mathbb{G})}
\end{equation}
holds for all $f\in \dot{L}^{p}_{a}(\mathbb{G})$ and $g\in \dot{L}^{q}_{b}(\mathbb{G})$, where $\dot{L}^{p}_{a}(\mathbb{G})$ stands for a homogeneous Sobolev space of order $a$ over $L^p$ on the graded Lie group $\mathbb{G}.$
\end{itemize}

\section{Stein-Weiss inequality on homogeneous group}
\label{SEC:2}
 Let us recall that a Lie group (on $\mathbb{R}^{N}$) $\mathbb{G}$ with the dilation
$$D_{\lambda}(x):=(\lambda^{\nu_{1}}x_{1},\ldots,\lambda^{\nu_{N}}x_{N}),\; \nu_{1},\ldots, \nu_{n}>0,\; D_{\lambda}:\mathbb{R}^{N}\rightarrow\mathbb{R}^{N},$$
which is an automorphism of the group $\mathbb{G}$ for each $\lambda>0,$
is called a {\em homogeneous (Lie) group}. For simplicity, throughout this paper we use the notation $\lambda x$ for the dilation $D_{\lambda}.$  The homogeneous dimension of the homogeneous group $\mathbb{G}$ is denoted by $Q:=\nu_{1}+\ldots+\nu_{N}.$
Also, in this note we denote a homogeneous quasi-norm on $\mathbb{G}$ by $|x|$, which
is a continuous non-negative function
\begin{equation}
\mathbb{G}\ni x\mapsto |x|\in[0,\infty),
\end{equation}
with the properties

\begin{itemize}
	\item[i)] $|x|=|x^{-1}|$ for all $x\in\mathbb{G}$,
	\item[ii)] $|\lambda x|=\lambda |x|$ for all $x\in \mathbb{G}$ and $\lambda>0$,
	\item[iii)] $|x|=0$ iff $x=0$.
\end{itemize}
Moreover, the following polarisation formula on homogeneous Lie groups will be used in our proofs:
there is a (unique)
positive Borel measure $\sigma$ on the
unit quasi-sphere
$
\mathfrak{S}:=\{x\in \mathbb{G}:\,|x|=1\},
$
so that for every $f\in L^{1}(\mathbb{G})$ we have
\begin{equation}\label{EQ:polar}
\int_{\mathbb{G}}f(x)dx=\int_{0}^{\infty}
\int_{\mathfrak{S}}f(ry)r^{Q-1}d\sigma(y)dr.
\end{equation}
The quasi-ball centred at $x \in \mathbb{G}$ with radius $R > 0$ can be defined by
\begin{equation}
B(x,R) := \{y \in\mathbb {G} : |x^{-1} y|< R\}.
\end{equation}
We refer to \cite{FS1} for the original appearance of such groups, and to \cite{FR} for a recent comprehensive treatment.

Let us consider the integral operator
\begin{equation}
I_{\lambda,|\cdot|}u(x)=\int_{\mathbb{G}}\frac{u(y)}{|y^{-1} x|^{\lambda}}dy,\,\,\,0<\lambda<Q.
\end{equation}
Note that when $Q>\alpha>0$ and $\lambda=Q-\alpha$ we get the Riesz potential $I_{\lambda,|\cdot|}=I_{Q-\alpha,|\cdot|}$.
First we give a short proof of a version of the Hardy-Littlewood-Sobolev inequality on $\mathbb{G}$.
\begin{thm}\label{Trsob}
Let $\mathbb{G}$ be a homogeneous group of homogeneous dimension $Q$ and let $|\cdot|$ be an arbitrary homogeneous quasi-norm on $\mathbb{G}$.
Let $1<p<q<\infty,\,\,0<\lambda<Q$, $\frac{1}{q}=\frac{1}{p}+\frac{\lambda}{Q}-1$, and $u\in L^{p}(\mathbb{G})$. Then we have
\begin{equation}\label{rieszsobolev}
\|I_{\lambda,|\cdot|}u\|_{L^{q}(\mathbb{G})}\leq C \|u\|_{L^{p}(\mathbb{G})},
\end{equation}
where $C$ is a positive constant independent of $u.$
\end{thm}
\begin{rem}
	With the assumptions of  Theorem \ref{Trsob} and $h\in L^{q'}(\mathbb{G}),$ we have the following Hardy-Littlewood-Sobolev inequality
	\begin{equation}
	\left|\int_{\mathbb{G}}\int_{\mathbb{G}}\frac{u(y)h(x)}{|y^{-1} x|^{\lambda}}dxdy\right|\leq C\|u\|_{L^{p}(\mathbb{G})}\|h\|_{L^{q'}(\mathbb{G})},
	\end{equation}
	where $C$ is a positive constant independent of $u$ and $h$. This gives \eqref{EQ:HLSi1}.
\end{rem}
\begin{proof}[Proof of Theorem \ref{Trsob}]
As in the Euclidean case we will show that there is a constant $C>0$, such that
\begin{equation}\label{needtoprove}
m\{x:|K*u(x)|>\zeta\}\leq C\frac{\|u\|_{L^{p}(\mathbb{G})}^{q}}{\zeta^{q}},
\end{equation}
where $m$ is the Haar measure on $\mathbb{G}$, $K(x)=|x|^{-\lambda}$ and $I_{\lambda,|\cdot|}u(x)=K*u(x)$, where $*$ is  convolution. This implies inequality \eqref{rieszsobolev} via the Marcinkiewicz interpolation theorem.
 Let $K(x)=K_{1}(x)+K_{2}(x)$, where
\begin{equation}\label{2.7}
K_{1}(x):=
 \begin{cases}
   K(x),\,\,\,\text{if}\,\,\,|x|\leq\mu, \\
   0,\,\,\,\text{if}\,\,\,|x|>\mu,
 \end{cases}
\text{and}\,\,\,\,
K_{2}(x):=
 \begin{cases}
   K(x),\,\,\,\text{if}\,\,\,|x|>\mu, \\
   0,\,\,\,\text{if}\,\,\,|x|\leq\mu.
 \end{cases}
\end{equation}
Here $\mu$ is a positive constant. We have $I_{\lambda,|\cdot|}u(x)=K*u(x)=K_{1}*u(x)+K_{2}*u(x)$, so
\begin{equation}\label{ocenkamery}
m\{x:|K*u(x)|>2\zeta\}\leq m\{x:|K_{1}*u(x)|>\zeta\}+m\{x:|K_{2}*u(x)|>\zeta\}.
\end{equation}
It is enough to prove inequality \eqref{needtoprove} with $2\zeta$ instead of $\zeta$ in the left-hand side of
the inequality. Without loss of generality we can assume $\|u\|_{L^{p}(\mathbb{G})}=1$ and by using Chebychev's and Minkowski's inequalities, we get
\begin{multline}\label{2.9}
m\{x:|K_{1}*u(x)|>\zeta\}\leq\frac{\int_{|K_{1}*u|>\zeta}|K_{1}*u|^{p}dx}{\zeta^{p}}\\
\leq\frac{\|K_{1}*u\|^{p}_{L^{p}(\mathbb{G})}}{\zeta^{p}}\leq\frac{\|K_{1}\|^{p}_{L^{1}(\mathbb{G})}\|u\|^{p}_{L^{p}(\mathbb{G})}}{\zeta^{p}}=\frac{\|K_{1}\|^{p}_{L^{1}(\mathbb{G})}}{\zeta^{p}}.
\end{multline}
 By using \eqref{EQ:polar} and \eqref{2.7}, we compute
\begin{multline}
\|K_{1}\|_{L^{1}(\mathbb{G})}=\int_{0<|x|\leq\mu}|x|^{-\lambda}dx=\int_{0}^{\mu}r^{Q-1}r^{-\lambda}dR\int_{\mathfrak{S}}|y|^{-\lambda}d\sigma(y)\\
=|\mathfrak{S}| \int_{0}^{\mu}r^{Q-\lambda-1}dr= \frac{|\mathfrak {S}|}{Q-\lambda}(r^{Q-\lambda}|^{\mu}_{0})=\frac{|\mathfrak {S}|}{Q-\lambda} \mu^{Q-\lambda},
\end{multline}
where $|\mathfrak{S}|$ is the $Q - 1$ dimensional surface measure of the unit
quasi-sphere $\mathfrak{S}$.
By using this in \eqref{2.9}, we obtain
\begin{equation}
m\{x:|K_{1}*u(x)|>\zeta\}\leq \left(\frac{|\mathfrak {S}|}{Q-\lambda}\right)^{p} \frac{\mu^{(Q-\lambda)p}}{\zeta^{p}}.
\end{equation}
Similarly by using Young's inequality, \eqref{EQ:polar} and the assumptions, we get
\begin{multline}
\|K_{2}*u\|_{L^{\infty}(\mathbb{G})}\leq\|K_{2}\|_{L^{p'}(\mathbb{G})}\|u\|_{L^{p}(\mathbb{G})}=\left(\int_{\mu}^{\infty}r^{-\lambda p'}r^{Q-1}dr\int_{\mathfrak{S}}|y|^{-\lambda p'}d\sigma(y)\right)^{\frac{1}{p'}}\\
 =\left(\frac{|\mathfrak{S}|}{Q-\lambda p'}\right)^{\frac{1}{p'}}\left(\int_{\mu}^{\infty}r^{Q-\lambda p'-1}dr\right)^{\frac{1}{p'}}=\left(\frac{|\mathfrak{S}|}{Q-\lambda p'}\right)^{\frac{1}{p'}} (r^{Q-\lambda p'}|^{\infty}_{\mu})^{\frac{1}{p'}}\\
 =\left(\frac{|\mathfrak{S}|}{\lambda p'-Q}\right)^{\frac{1}{p'}}\mu^{-\frac{Q}{q}},
\end{multline}
since from the assumptions, we get $\frac{Q-\lambda p'}{p'}=\frac{Q}{p'}-\lambda=Q(1-\frac{1}{p}-\frac{\lambda}{Q})=-\frac{Q}{q}$. Moreover, if $\left(\frac{|\mathfrak{S}|}{\lambda p'-Q}\right)^{\frac{1}{p'}}\mu^{-\frac{Q}{q}}=\zeta$, then $\mu=\left(\frac{|\mathfrak{S}|}{\lambda p'-Q}\right)^{-\frac{q}{Qp'}} \zeta^{-\frac{\theta}{Q}}$, so we have $\|K_{2}*u\|_{L^{\infty}(\mathbb{G})}\leq\zeta$.  Thus, we have $m\{x:|K_{2}*u|>\zeta\}=0.$
Combining these facts with \eqref{ocenkamery}, $\|u\|_{L^{p}(\mathbb{G})}=1$ and the assumptions  we establish
\begin{multline}
m\{x:|K*u|>2\zeta\}
\leq \left(\frac{|\mathfrak {S}|}{Q-\lambda}\right)^{p}\frac{\mu^{(Q-\lambda)p}}{\zeta^{p}}\\
=\left(\frac{|\mathfrak {S}|}{Q-\lambda}\right)^{p}\left(\frac{|\mathfrak{S}|}{\lambda p'-Q}\right)^{-\frac{q(Q-\lambda)p}{Qp'}}\zeta^{\frac{-(Q-\lambda)pq}{Q}-p}
\leq C\zeta^{\frac{-(Q-\lambda)pq}{Q}-p}=C\zeta^{(\frac{\lambda}{Q}-1)pq-p}\\
=C\zeta^{(\frac{1}{q}-\frac{1}{p})pq-p}=C\zeta^{p-q-p}=C\frac{\|u\|^{q}_{L^{p}(\mathbb{G})}}{\zeta^{q}}.
\end{multline}
For completeness, let us recall two well-known ingredients.
\begin{defn}[\cite{steinbook}]\label{defin}
Let $ 1\leq p\leq\infty$, $1\leq q<\infty$ and  $V:L^{p}(\mathbb{G})\rightarrow L^{q}(\mathbb{G})$ be a operator, then $V$ is called an operator of $\textit{weak type}$ $(p,q)$ if
\begin{equation}
m\{x:|Vu|>\zeta\}\leq C\left(\frac{\|u\|_{L^{p}(\mathbb{G})}}{\zeta}\right)^{q},\,\,\,\,\zeta>0,
\end{equation}
where $C$ is a positive constant and independent by $u$.
\end{defn}
Let us also recall the classical Marcinkiewicz interpolation theorem:
\begin{thm}
Let $V$ be sublinear operator of weak type $(p_{k}, q_{k})$ with $1 \leq p_{k} \leq q_k< \infty$, $k = 0, 1$ and $q_0 < q_1$.
Then $V$ is bounded from $L^{p}(\mathbb{G})$ to $L^{q}(\mathbb{G})$ with
 \begin{equation}
 \frac{1}{p}=\frac{1-\gamma}{p_{0}}+\frac{\gamma}{p_{1}},\,\,\,\frac{1}{q}=\frac{1-\gamma}{q_{0}}+\frac{\gamma}{q_{1}},
 \end{equation}
 for any $0<\gamma<1$, namely,
 \begin{equation}
 \|Vu\|_{L^{q}(\mathbb{G})}\leq C \|u\|_{L^{p}(\mathbb{G})},
 \end{equation}
 for any $u\in L^{p}(\mathbb{G})$ and $C$ is a positive constant.

 \end{thm}

From assumptions $\frac{1}{q}=\frac{1}{p}+\frac{\lambda}{Q}-1<\frac{1}{p}$, then $q>p$. According to Definition \ref{defin}, $I_{\lambda,|\cdot|}u$ is of weak type $(p,q)$, so by using the Marcinkiewicz interpolation theorem, we prove \eqref{rieszsobolev}.

The proof of Theorem \ref{Trsob} is complete.
\end{proof}

The following statements will be useful to prove the homogeneous group version of the Stein-Weiss inequality (\cite[Theorem B*]{StWe58}). The next proposition is well-known, see e.g. {\cite[Theorem 3.1.39 and Proposition 3.1.35]{FR}} and historical references therein.

\begin{prop}\label{prop_quasi_norm}
Let $\mathbb{G}$ be a homogeneous Lie group. Then there exists a homogeneous
quasi-norm on $\mathbb{G}$ which is a norm, that is, a homogeneous quasi-norm $|\cdot|$
which satisfies the triangle inequality
\begin{equation}
|x y|\leq |x| + |y|, \,\,\,\forall x, y \in \mathbb{G}.
\end{equation}
Furthermore, all homogeneous quasi-norms on $\mathbb{G}$ are equivalent.
\end{prop}

The next theorem is the integral version of Hardy inequalities on general homogeneous groups that will be instrumental in our proof.

\begin{thm}[\cite{RY}]\label{integral_hardy}
Let $\mathbb{G}$ be a homogeneous
group of homogeneous dimension $Q$ and let $1 < p \leq q < \infty$. Let $W(x)$ and $U(x)$,  be
positive functions on $\mathbb{G}$. Then we have the following properties:

(1) The inequality
\begin{equation}\label{5.2}
\left(\int_{\mathbb{G}}\left(\int_{B(0,|x|)}f(z)dz\right)^{q}W(x)dx\right)^{\frac{1}{q}}\leq C_{1} \left(\int_{\mathbb{G}}f^{p}(x)U(x)dx\right)^{\frac{1}{p}}
\end{equation}
holds for all $f\geq0$ a.e. on $\mathbb{G}$ if only if
\begin{equation}\label{5.2.1}
A_{1}:=\sup_{R>0}\left(\int_{\mathbb{G}\setminus B(0,|x|)}W(x)dx\right)^{\frac{1}{q}}\left(\int_{B(0,|x|)}U^{1-p'}(x)dx\right)^{\frac{1}{p'}}<\infty.
\end{equation}
(2) The inequality
\begin{equation}\label{5.4}
\left(\int_{\mathbb{G}}\left(\int_{\mathbb{G}\setminus B(0,|x|)}f(z)dz\right)^{q}W(x)dx\right)^{\frac{1}{q}}\leq C_{2} \left(\int_{\mathbb{G}}f^{p}(x)U(x)dx\right)^{\frac{1}{p}},
\end{equation}
holds for all $f \geq 0$ if and only if
\begin{equation}\label{5.4.1}
A_{2}:=\sup_{R>0}\left(\int_{B(0,|x|)}W(x)dx\right)^{\frac{1}{q}}\left(\int_{\mathbb{G}\setminus B(0,|x|)}U^{1-p'}(x)dx\right)^{\frac{1}{p'}}<\infty.
\end{equation}
(3) If $\{C_i\}^{2}_{i=1}$ are the smallest constants for which \eqref{5.2} and \eqref{5.4} hold, then
\begin{equation}
A_{i} \leq C_{i} \leq (p')^{\frac{1}{p'}}p^{\frac{1}{q}} A_{i}, \,\,\,i = 1, 2.
\end{equation}
\end{thm}

Now we formulate the Stein-Weiss inequality on $\mathbb{G}$.
\begin{thm}\label{stein-weiss3}
Let $\mathbb{G}$ be a homogeneous group of homogeneous dimension $Q$ and let $|\cdot|$ be an arbitrary homogeneous quasi-norm on $\mathbb{G}$.
Let $0<\lambda<Q$, $1<p<\infty$, $\alpha<\frac{Q}{p'}$, $\beta<\frac{Q}{q}$, $\alpha+\beta\geq0$, $\frac{1}{q}=\frac{1}{p}+\frac{\alpha+\beta+\lambda}{Q}-1$, where $\frac{1}{p}+\frac{1}{p'}=1$ and $\frac{1}{q}+\frac{1}{q'}=1$. Then for $1<p\leq q<\infty$, we have
\begin{equation}\label{stein-weiss}
\||x|^{-\beta}I_{\lambda,|\cdot|}u\|_{L^{q}(\mathbb{G})}\leq C \||x|^{\alpha}u\|_{L^{p}(\mathbb{G})}.
\end{equation}
where $C$ is positive constant  and independent by $u$.
\end{thm}

	In inequality \eqref{stein-weiss} with $\alpha=0$ we get the weighted Hardy-Littlewood-Sobolev inequality established in  \cite[Theorem 4.1]{RY}. Thus, by setting $\alpha=\beta=0$ we get Hardy-Littlewood-Sobolev inequality on the homogeneous Lie groups. In the Abelian (Euclidean) case ${\mathbb G}=(\mathbb R^{N},+)$, we have $Q=N$ and $|\cdot|$ can be any homogeneous quasi-norm  on $\mathbb R^{N}$, so with the usual Euclidean distance, i.e. $|\cdot|=\|\cdot\|_{E}$, Theorem \ref{stein-weiss3} gives the classical result of Stein and Weiss (Theorem \ref{Classiacal_Stein-Weiss_inequality}).

\begin{proof}[Proof of Theorem \ref{stein-weiss3}]
Define
\begin{equation}
\||x|^{-\beta}I_{\lambda,|\cdot|}u\|^{q}_{L^{q}(\mathbb{G})}=\int_{\mathbb{G}}\left(\int_{\mathbb{G}}\frac{u(y)}{|x|^{\beta}|y^{-1} x|^{\lambda}}dy\right)^{q}dx=I_{1}+I_{2}+I_{3},
\end{equation}
where
\begin{equation}
I_{1}=\int_{\mathbb{G}}\left(\int_{B\left(0,\frac{|x|}{2}\right)}\frac{u(y)}{|x|^{\beta}|y^{-1} x|^{\lambda}}dy\right)^{q}dx,
\end{equation}
\begin{equation}
I_{2}=\int_{\mathbb{G}}\left(\int_{B(0,2|x|)\setminus B\left(0,\frac{|x|}{2}\right)}\frac{u(y)}{|x|^{\beta}|y^{-1} x|^{\lambda}}dy\right)^{q}dx,
\end{equation}
and
\begin{equation}
I_{3}=\int_{\mathbb{G}}\left(\int_{\mathbb{G}\setminus B(0,2|x|)}\frac{u(y)}{|x|^{\beta}|y^{-1} x|^{\lambda}dy}\right)^{q}dx.
\end{equation}

From now on, in view of Proposition \ref{prop_quasi_norm} we can assume that our quasi-norm is actually a norm.

\textbf{Step 1.} Let us consider $I_{1}$.   By using Proposition \ref{prop_quasi_norm} and the properties of the quasi-norm with  $|y|\leq\frac{|x|}{2}$, we get
$$|x|=|x^{-1}|=|x^{-1}y y^{-1}|$$ $$\leq |x^{-1} y|+|y^{-1}|=|y^{-1} x|+|y|$$
$$\leq |y^{-1} x|+\frac{|x|}{2}.$$
Then for any $\lambda>0$, we have
$$2^{\lambda}|x|^{-\lambda}\geq |y^{-1} x|^{-\lambda}.$$
Therefore, we get
\begin{multline}
I_{1}=\int_{\mathbb{G}}\left(\int_{B\left(0,\frac{|x|}{2}\right)}\frac{u(y)}{|x|^{\beta}|y^{-1} x|^{\lambda}}dy\right)^{q}dx\leq 2^{\lambda}\int_{\mathbb{G}}\left(\int_{B\left(0,\frac{|x|}{2}\right)}\frac{u(y)}{|x|^{\beta+\lambda}}dy\right)^{q}dx\\
=2^{\lambda}\int_{\mathbb{G}}\left(\int_{B\left(0,\frac{|x|}{2}\right)}u(y)dy\right)^{q}|x|^{-(\beta+\lambda)q}dx.
\end{multline}
If condition \eqref{5.2.1} in Theorem \ref{integral_hardy} with $W(x)=|x|^{-(\beta+\lambda)q}$ and $U(y)=|y|^{\alpha p}$ in \eqref{5.2} is satisfied, then we have
\begin{equation}
I_{1}\leq2^{\lambda}\int_{\mathbb{G}}\left(\int_{B(0,\frac{|x|}{2})}u(y)dy\right)^{q}|x|^{-(\beta+\lambda)q}dx
\leq C_{1}\||x|^{\alpha}u\|^{q}_{L^{p}(\mathbb{G})}.
\end{equation}
Let us verify condition \eqref{5.2.1}. So from the assumption we have $\alpha<\frac{Q}{p'}$, then $$\frac{1}{q}=\frac{1}{p}+\frac{\alpha+\beta+\lambda}{Q}-1<\frac{1}{p}+\frac{\frac{Q}{p'}+\beta+\lambda}{Q}-1=\frac{1}{p}+\frac{1}{p'}+\frac{\beta+\lambda}{Q}-1=\frac{\beta+\lambda}{Q},$$
that is, $Q-(\beta+\lambda)q<0$ and by the using polar decomposition \eqref{EQ:polar}:
\begin{multline}
\left(\int_{\mathbb{G}\setminus B(0,|x|)}W(x)dx\right)^{\frac{1}{q}}=\left(\int_{\mathbb{G}\setminus B(0,|x|)}|x|^{-(\beta+\lambda)q}dx\right)^{\frac{1}{q}}\\
=\left(\int_{R}^{\infty}\int_{\mathfrak{S}}r^{Q-1}r^{-(\beta+\lambda)q}drd\sigma(y)\right)^{\frac{1}{q}}
=\left(|\mathfrak{S}|\int_{R}^{\infty}r^{Q-1-(\beta+\lambda)q}dr\right)^{\frac{1}{q}}\leq C R^{\frac{Q-(\beta+\lambda)q}{q}}.
\end{multline}
Since $\alpha<\frac{Q}{p'}$, we have
$$\alpha p(1-p')+Q>\alpha p(1-p')+\alpha p'=\alpha p+\alpha p'(1-p)=\alpha p -\alpha p=0 .$$
So, $\alpha p(1-p')+Q>0$. Then, let us consider
\begin{multline}
\left(\int_{ B(0,|x|)}U^{1-p'}(x)dx\right)^{\frac{1}{p'}}=\left(\int_{ B(0,|x|)}|x|^{(1-p')\alpha p}dx\right)^{\frac{1}{p'}}\\=\left(\int^{ R}_{0}\int_{\mathfrak{S}}r^{(1-p')\alpha p}r^{Q-1}drd\sigma(y)\right)^{\frac{1}{p'}}
\leq C\left(|\mathfrak{S}|\int^{ R}_{0}r^{(1-p')\alpha p+Q-1}dr\right)^{\frac{1}{p'}}\\ \leq C R^{\frac{(1-p')\alpha p+Q}{p'}}=CR^{\frac{Q-\alpha p'}{p'}}.
\end{multline}
Moreover, the assumptions imply
$$A_{1}=\sup_{R>0}\left(\int_{\mathbb{G}\setminus B(0,|x|)}W(x)dx\right)^{\frac{1}{q}}\left(\int_{ B(0,|x|)}U^{1-p'}(x)dx\right)^{\frac{1}{p'}}\leq C R^{\frac{Q-(\beta+\lambda)q}{q}+\frac{Q-\alpha p'}{p'}}$$
$$=C R^{Q(\frac{1}{q}-\frac{1}{p}-\frac{\alpha+\beta+\lambda}{Q}+1)}=C<\infty,$$
where $C=C(\alpha,\beta,p,\lambda)$ is a positive constant.
Then by using \eqref{5.2}, we obtain
\begin{equation}
I_{1}\leq C\int_{\mathbb{G}}\left(\int_{B\left(0,\frac{|x|}{2}\right)}u(y)dy\right)^{q}|x|^{-(\beta+\lambda)q}dx
\leq C_{1}\||x|^{\alpha}u\|^{q}_{L^{p}(\mathbb{G})}.
\end{equation}

\textbf{Step 2.} As in the previous case $I_{1}$, now we consider $I_{3}$. From $2|x|\leq |y|$,  we calculate
$$|y|=|y^{-1}|=|y^{-1} x x^{-1}|\leq |y^{-1} x|+|x|$$
$$\leq |y^{-1} x|+\frac{|y|}{2},$$
that is,
$$\frac{|y|}{2}\leq |y^{-1} x|.$$
Then, if condition \eqref{5.4.1} with $W(x)=|x|^{-\beta q}$ and $U(y)=|y|^{(\alpha+\lambda)p}$ is satisfied, then we have
\begin{multline}
I_{3}=\int_{\mathbb{G}}\left(\int_{\mathbb{G}\setminus B(0,2|x|)}\frac{u(y)}{|x|^{\beta}|y^{-1}x|^{\lambda}}dy\right)^{q}dx\leq C\int_{\mathbb{G}}\left(\int_{\mathbb{G}\setminus B(0,2|x|)}\frac{u(y)}{|x|^{\beta}|y|^{\lambda}}dy\right)^{q}dx\\
=C\int_{\mathbb{G}}\left(\int_{\mathbb{G}\setminus B(0,2|x|)}u(y)|y|^{-\lambda}dy\right)^{q}|x|^{-\beta q}dx\leq C\||x|^{\alpha}u\|^{q}_{L^{p}(\mathbb{G})}.
\end{multline}
Now let us check condition \eqref{5.4.1}. We have
\begin{multline}
\left(\int_{ B(0,|x|)}W(x)dx\right)^{\frac{1}{q}}=\left(\int_{ B(0,|x|)}|x|^{-\beta q}dx\right)^{\frac{1}{q}}\\=\left(\int_{0}^{R}\int_{\mathfrak{S}}r^{-\beta q}r^{Q-1}drd\sigma(y)\right)^{\frac{1}{q}}
\leq C R^{\frac{Q-\beta q}{q}},
\end{multline}
where $Q-\beta q>0$,
and
\begin{multline}
\left(\int_{\mathbb{G}\setminus B(0,|x|)}U^{1-p'}(x)dx\right)^{\frac{1}{p'}}=\left(\int_{\mathbb{G}\setminus B(0,|x|)}|x|^{(\alpha+\lambda)(1-p')p}dx\right)^{\frac{1}{p'}}\\=\left(\int_{R}^{\infty}\int_{\mathfrak{S}}r^{Q-1}r^{(\alpha+\lambda)(1-p')p}drd\sigma(y)\right)^{\frac{1}{p'}}
\leq C R^{\frac{Q-p'(\alpha+\lambda)}{p'}},
\end{multline}
where from  $\beta<\frac{Q}{q}$, we obtain $Q-p'(\alpha+\lambda)<0$.

Combining these facts we have
\begin{multline}
A_{2}:=\sup_{R>0}\left(\int_{ B(0,|x|)}W(x)dx\right)^{\frac{1}{\theta}}\left(\int_{\mathbb{G}\setminus B(0,|x|)}U^{1-p'}(x)dx\right)^{\frac{1}{p'}}\leq C R^{\frac{Q-p'(\alpha+\lambda)}{p'}+\frac{Q-\beta q}{q}}\\
=C R^{\frac{Q}{p'}-(\alpha+\beta+\lambda)+\frac{Q}{q}}=C R^{Q(\frac{1}{p'}-\frac{\alpha+\beta+\lambda}{Q}+\frac{1}{q})}=C<\infty,
\end{multline}
where $C=C(\alpha,\beta,p,\lambda)$ is a positive constant. Then we establish
\begin{equation}
I_{3}=\int_{\mathbb{G}}\left(\int_{\mathbb{G}\setminus B(0,2|x|)}\frac{u(y)}{|x|^{\beta}|y^{-1} x|^{\lambda}}dy\right)^{q}dx\leq C\||x|^{\alpha}u\|^{q}_{L^{p}(\mathbb{G})}.
\end{equation}

\textbf{Step 3.} Let us estimate $I_{2}$ now.

\textbf{Case 1:} $p<q$. From $\frac{|x|}{2}<|y|<2|x|$, we obtain
$$\frac{|y^{-1} x|}{2}\leq \frac{|x|+|y|}{2}= \frac{|x|}{2}+\frac{|y|}{2}<\frac{3}{2}|y|,$$
that is, $$|y^{-1} x|<3|y|.$$
For all $\alpha+\beta\geq0$, we have
$$|y^{-1} x|^{\alpha+\beta}< 3^{\alpha+\beta}|y|^{\alpha+\beta}= 3^{\alpha+\beta}|y|^{\alpha}|y|^{\beta}\leq 3^{\alpha+\beta}2^{|\beta|}|x|^{\beta}|y|^{\alpha}.$$
Therefore,
\begin{multline*}
I_{2}=\int_{\mathbb{G}}\left(\int_{B(0,2|x|)\setminus B\left(0,\frac{|x|}{2}\right)}\frac{u(y)}{|x|^{\beta}|y^{-1} x|^{\lambda}}dy\right)^{q}dx\\ \leq C\int_{\mathbb{G}}\left(\int_{B(0,2|x|)\setminus B\left(0,\frac{|x|}{2}\right)}\frac{|y|^{\alpha}u(y)}{|y^{-1} x|^{\alpha+\beta+\lambda}}dy\right)^{q}dx
\\ \leq C \int_{\mathbb{G}}\left(\int_{\mathbb{G}}\frac{|y|^{\alpha}u(y)}{|y^{-1}x|^{\alpha+\beta+\lambda}}dy\right)^{q}dx
=C\|I_{\lambda+\alpha+\beta,|\cdot|}\tilde{u}\|^{q}_{L^{q}(\mathbb{G})},
\end{multline*}
where $\tilde{u}(x)=|x|^{\alpha}u(x)$.

By assumption $\frac{1}{q}-\frac{1}{p}=\frac{\lambda+\alpha+\beta}{Q}-1<0$, then $Q>\lambda+\alpha+\beta$ and by using Theorem \ref{Trsob} with $p<q$,  we establish
\begin{equation}
I_{2}\leq C\|I_{\lambda+\alpha+\beta,|\cdot|}\tilde{u}\|^{q}_{L^{q}(\mathbb{G})}\leq C\|\tilde{u}\|_{L^{p}(\mathbb{G})}^{q}=C\||x|^{\alpha}u\|_{L^{p}(\mathbb{G})}^{q}.
\end{equation}
\textbf{Case 2:} $p=q$. We decompose $I_{2}$ as
\begin{equation}
I_{2}=\sum_{k\in \mathbb{Z}}\int_{2^{k}\leq |x| \leq 2^{k+1}}\left(\int_{B(0,2|x|)\setminus B\left(0,\frac{|x|}{2}\right)}\frac{u(y)}{|x|^{\beta}|y^{-1} x|^{\lambda}}dy\right)^{p}dx.
\end{equation}

From $|x| \leq 2|y| \leq 4 |x|$ and $2^{k} \leq |x| \leq 2^{k+1}$, we have $2^{k-1} \leq |y| \leq 2^{k+2}$ and  $0 \leq |y^{-1} x| \leq 3|x| \leq 3 \cdot 2^{k+1}$.

 By using  Young's inequality with $\frac{1}{p}+\frac{1}{r}=1+\frac{1}{q}$ (our case $p=q$, hence $r=1$), we calculate
\begin{align*}
I_{2}&=\sum_{k\in \mathbb{Z}}\int_{2^{k}\leq |x| \leq 2^{k+1}}\left(\int_{B(0,2|x|)\setminus B\left(0,\frac{|x|}{2}\right)}\frac{u(y)}{|x|^{\beta}|y^{-1} x|^{\lambda}}dy\right)^{p}dx\\&
=\sum_{k\in \mathbb{Z}}\int_{2^{k}\leq |x| \leq 2^{k+1}}\left(\int_{B(0,2|x|)\setminus B\left(0,\frac{|x|}{2}\right)}\frac{u(y)}{|y^{-1} x|^{\lambda}}dy\right)^{p}\frac{dx}{|x|^{\beta p}}\\&
\leq \sum_{k\in \mathbb{Z}} 2^{-\beta p k}\|u\cdot\chi_{\{2^{k-1}\leq |y|\leq 2^{k+2}\}}*|x|^{-\lambda}\|^{p}_{L^{p}(\mathbb{G})}\\&
\leq \sum_{k\in \mathbb{Z}} 2^{-\beta p k} \||x|^{-\lambda}\cdot\chi_{\{0\leq |y|\leq 3\cdot2^{k+1}\}}\|^{p}_{L^{1}(\mathbb{G})}\|u\cdot\chi_{\{2^{k-1}\leq |y|\leq 2^{k+2}\}}\|^{p}_{L^{p}(\mathbb{G})}\\&
\leq C \sum_{k\in \mathbb{Z}} 2^{(Q-\lambda-\beta)kp}\|u\cdot\chi_{\{2^{k-1}\leq |y|\leq 2^{k+2}\}}\|^{p}_{L^{p}(\mathbb{G})} =C \sum_{k\in \mathbb{Z}} 2^{\alpha kp}\|u\cdot\chi_{\{2^{k-1}\leq |y|\leq 2^{k+2}\}}\|^{p}_{L^{p}(\mathbb{G})}\\&
=C \sum_{k\in \mathbb{Z}} \|2^{\alpha (k-1)}u\cdot\chi_{\{2^{k-1}  \leq |y|\leq 2^{k+2}\}}\|^{p}_{L^{p}(\mathbb{G})} \leq
 C\sum_{k\in \mathbb{Z}}\||y|^{\alpha}u\cdot\chi_{\{2^{k-1}\leq |y|\leq 2^{k+2}\}}\|^{p}_{L^{p}(\mathbb{G})}\\&
 = C\||x|^{\alpha}u\|^{p}_{L^{p}(\mathbb{G})}.
\end{align*}
Theorem \ref{stein-weiss3} is proved.
\end{proof}
\begin{rem}
With assumptions Theorem \ref{stein-weiss3} and $h\in L^{q'}(\mathbb{G})$, we have the following Stein-Weiss inequality
\begin{equation}
\left|\int_{\mathbb{G}}\frac{u(y)h(x)}{|x|^{\beta}|y^{-1}x|^{\lambda}|y|^{\alpha}}dxdy\right|\leq C\|u\|_{L^{p}(\mathbb{G})}\|h\|_{L^{q'}(\mathbb{G})},
\end{equation}
where $C$ is a positive constant independent of $u$ and $h$. This gives \eqref{EQ:HLSi2}.
\end{rem}

\end{document}